\documentclass[11pt]{amsart}

\date{}
\usepackage{amsmath,amsthm,amsfonts,amssymb}
 
\newcommand{\VC}{\text{VC}}
\newcommand{\LD}{\text{LD}}
\newcommand{\coVC}{\text{coVC}}
\newcommand{\VCd}{\text{VCd}}
\newcommand{\VCm}{\text{VCm}}
\newcommand{\VCId}{\text{VCd}_{\text{ind}}}
\newcommand{\VCI}{\text{VC}_\text{ind}}
\newcommand{\monster}{\mathfrak{M}}
\newcommand{\M}[1]{\text{M}^{|#1|}}
\renewcommand{\phi}{\varphi}
\renewcommand{\bf}[1]{\textbf{#1}}
\newcommand{\x}{\textbf{x}}
\newcommand{\y}{\textbf{y}}
\newcommand{\sse}{\subseteq}

\newtheorem{theorem}{Theorem}[section]
\newtheorem{lemma}[theorem]{Lemma}
\newtheorem{proposition}[theorem]{Proposition}
\newtheorem{corollary}[theorem]{Corollary}
\newtheorem{definition}{Definition}[section] 

\title[$\VC$-density, indiscernibles, and the maximum property]{Vapnik-Chervonenkis density on indiscernible sequences, stability, and the maximum property}
\author{Hunter Johnson}
\email{hujohnson@jjay.cuny.edu}
 \address{Dept. Math \& CS, John Jay College, CUNY, 444 W. 59th St., New York, NY 10019.}

\begin{document} 

\begin{abstract}
This paper presents some finite combinatorics of set systems with applications to model theory, particularly the study of dependent theories.  There are two main results.  First, we give a way of producing lower bounds on $\VCI$-density, and use it to compute the exact $\VCI$-density of polynomial inequalities, and a variety of geometric set families.  The main technical tool used is the notion of a maximum set system, which we juxtapose to indiscernibles.  In the second part of the paper we give a maximum set system analogue to Shelah's characterization of stability using indiscernible sequences. 
\end{abstract} 
 \maketitle

\section{Introduction}

In the recent past there have been a number of papers relating various measures of the combinatorial structure of NIP theories to one another \cite{AsDoHaMaSt11,GuHi11, ItOnUs11}. One fact which emerged from this is the close relation of dp-rank and VC-density restricted to indiscernible sequences.  Guingona and Hill have used the term $\VCI$-density to describe VC-density restricted to indiscernibles.  At the end of their paper \cite{GuHi11}, Guingona and Hill ask if there is a useful characterization of when a formula has $\VCI$-density equal to VC-density.  We offer such a characterization below (Corollary \ref{C1}), and use it to compute the exact $\VCI$-density of certain formulas. 

A separate goal of this paper is to show how maximum set systems can in many cases be used as more accessible surrogates for indiscernible sequences.  To this end we translate Shelah's well-known characterization of stability in terms of indiscernible sets to a version involving maximum set systems.


\section{Notation}
 
  Let there be a fixed complete theory $T$, with a large saturated model $\monster = \langle \text{M},\ldots \rangle$.  All sets and models, unless otherwise stated, are assumed to be elementarily embedded in the model $\monster$. We write formulas in partitioned form $\phi(\x,\y)$, where $\x = \langle x_1,\ldots,x_l \rangle$, and $\y = \langle y_1,\ldots,y_k \rangle$.  We use $\mathcal{P}(X)$ to denote the power set of $X$.
  
  For $A \sse \M{\y}$ and $\bf{b} \in \M{\x}$, 
  
  $$\phi({\bf{b}},A) := \{\bf{a} \in A: \models \phi(\bf{b},\bf{a})\}$$

We use $S_\phi(A)$ for $A \sse \M{\y}$ to denote the set of $\phi$-types over $A$, where a $\phi$-type over $A$ is a maximal consistent set of the form $\{\pm \phi(\x,\bf{a}):\bf{a} \in \M{\y}\}$. 
   We let $S_\phi(A)|_B = \{{tp}_\phi(\bf{b}/A):\bf{b} \in B\}$ when $B \sse \M{\x}$.
        For an individual type $p \in S_\phi(A)$, we often identify $p$ and the set of its positive parameter instances $ \{\bf{a}\in A:  \phi(\x,\bf{a}) \in p\}$ without further comment.  Similarly we sometimes identify $S_\phi(A)|_B$ and $\{\phi(\bf{b},A): \bf{b} \in B\}$, as in the following definition. 

\begin{definition}
Let $\mathcal{C} \subseteq \mathcal{P}(\M{\x})$. We will say that $\phi(\x,\y)$ \emph{traces} $\mathcal{C}$ if for some $A \sse \M{\x}$, $\mathcal{C} \subseteq S_\phi(A).$

\end{definition}

We now give some purely combinatorial definitions.  For the rest of this section suppose $X$ is a set and $\mathcal{C} \sse \mathcal{P}(X)$.  For $A \sse X$, let $\mathcal{C}|^A = \{C \cap A: C \in \mathcal{C}\}$.  Say that $\mathcal{C}$ \emph{shatters} $A$ if $\mathcal{C}|^A = \mathcal{P}(A)$. 

\begin{definition} 
  The \emph{Vapnik-Chervonenkis (VC) dimension} of $\mathcal{C}$, denoted $\VC(\mathcal{C})$, is $|A|$ where $A\sse X$ is of maximum finite cardinality such that $\mathcal{C}$ shatters $A$.  
\end{definition}

If the VC-dimension of $\mathcal{C}$ does not exist, we write $\VC(\mathcal{C}) = \infty$.  The VC-dimension was first considered in \cite{VaCh71} and was introduced into model theory by Laskowski \cite{L92}. The following notion of a maximum VC family was defined by Welzl \cite{We87}.

\begin{definition}
  Say that $\mathcal{C}$ is \emph{$d$-maximum} for $d \in \omega$ if for any finite $A \sse X$, $|\mathcal{C}|^A| = {{|A|}\choose{\leq d}}$.
\end{definition}
Here ${{n}\choose{\leq k}}$ is shorthand for $\sum_{i=0}^k {{n}\choose{i}}$ if $k<n$ and $2^n$ otherwise.  

\begin{lemma}[Sauer's Lemma \cite{Sa72,Sh72}] 
If $\mathcal{C}$ has $\VC(\mathcal{C}) = d$, then for any finite $A \sse X$, 
$$|\mathcal{C}|^A| \leq {{|A|}\choose{\leq d}}$$ 
\end{lemma}

Thus a $d$-maximum set system is ``extremal" among set systems of VC-dimension $d$.  These set systems are highly structured, and well-understood \cite{Fl89,KW}.  There are several examples that arise from natural algebraic situations.  In fact it is conjectured that all $d$-maximum set systems arise from (or embed naturally in) arrangements of half-spaces, either in a euclidean or hyperbolic space \cite{BIP}.

It is easy to see that if $\mathcal{C}$ is $d$-maximum on $X$, and $A \sse X$ has $|A|=d+1$, then
$$\mathcal{C}|^A = \mathcal{P}(A) \setminus \{C\}$$
for some $C \sse A$.  Floyd \cite{Fl89} calls such a $C$ the \emph{forbidden label} of $\mathcal{C}$ on $A$. 

Let $[X]^m := \{A \sse X: |A|=m\}$. For a fixed $d$-maximum $\mathcal{C} \sse \mathcal{P}(X)$, associate with each $A \in [X]^{d+1}$ the forbidden label $C_A \sse A$, where $\mathcal{C}|^A = \mathcal{P}(A) \setminus \{C_A\}$.

  Floyd proves the following.
  
  \begin{proposition}\label{Pf}
   On a finite domain $X$, any $d$-maximum set system $\mathcal{C}$ is \emph{characterized} by its forbidden labels, in the sense that $\forall B \sse X$,
$$B \in \mathcal{C} \iff \forall A \in [X]^{d+1}(B\cap A) \neq C_A $$
\end{proposition}
\begin{proof}
Left to right is obvious.  For right to left, let $B$ satisfy the given conditions.  Then $\mathcal{C} \cup \{B\}$ shatters no sets not shattered by $\mathcal{C}$.  By Sauer's Lemma, $B$ must already be in $\mathcal{C}$.
\end{proof}

We now define the notion of a forbidden code, which is essentially the ``form" of a forbidden label, when an ordering is present.

Let $\mathfrak{L}_\mathcal{C}(X) = \{C_A: A \in [X]^{d+1}\}$ denote the forbidden labels of $\mathcal{C}$ on $X$.  Let $<$ be a fixed but arbitrary linear order on $X$. For each $C_A \in \mathfrak{L}_\mathcal{C}(X)$, let $\overline{C_A} = \langle t_0,\ldots,t_d \rangle$, where each $t_i \in \{0,1\}$, and $t_i =1$ if and only if the $i$th element of $A$ is in $C_A$. Define $\overline{\mathfrak{L}_\mathcal{C}(X)} = \{\overline{C_A}:C_A \in \mathfrak{L}_\mathcal{C}(X)\}$.

We will refer to  $\overline{\mathfrak{L}_\mathcal{C}(X)}$ as the set of \emph{forbidden codes} on $X$ for $\mathcal{C}$, with respect to $<$.  
When a maximum set system has a unique forbidden code, that code determines everything about the system at the level of finite traces.  Technically we say that the system is finitely characterized by the code.

\begin{definition}
The set system $\mathcal{C} \subseteq \mathcal{P}(X)$ is \textit{finitely characterized} by $\eta \in 2^{d+1}$ if for any finite $X_0 \subseteq X$, and $A \subseteq X_0$ the following are equivalent
\begin{enumerate}
\item $A \in \mathcal{C} |^{X_0}$
\item There are not elements $a_0 < \cdots < a_{d}$ in $X_0$ such that $a_i \in A \iff \eta(i) = 1$. 
\end{enumerate}
\end{definition}

There is a natural way in which forbidden codes can serve as combinatorial invariants for finite unions of points and $<$-convex subsets in $X$. To see this, suppose $(X,<)$ is a dense and complete linear order, and $B\sse X$ is a finite union of convex subsets.  Let $d$ be the number of boundary points of $B$. We can imagine that $B$ is defined by some $L_< = \{<\}$ formula $\psi(x,c_1,\ldots,c_d)$ with $c_1<\cdots<c_d \in X$. Define $\mathcal{F}(B) = \{\psi(X,c'_1,\cdots,c'_d): c'_1 < \cdots < c'_d \in X\}$.  Intuitively the elements of $\mathcal{F}(B) \sse \mathcal{P}(X)$ are the ``homeomorphic images" of $B$ in $(X,<)$. 
In \cite{J11} we show that $\mathcal{F}(B)$ is finitely characterized by some $\eta \in 2^{d+1}$.

 Define the \textit{genus} of $B$, denoted $\mathbb{G}(B)$, to be the $\eta \in 2^{d+1}$ that finitely characterizes $\mathcal{F}(B)$.  Equivalently, define $\mathbb{G}(B)$ to be any $\eta \in 2^{d+1}$ such that there are no $a_0 < \cdots < a_{d}$ in $X$ such that $a_i \in B$ if and only if $\eta(i)=1$, for all $i=0,\ldots,d$.

In \cite{J11} we show that such an $\eta$ exists, and is unique, as well as the further fact that for any $\eta \in 2^{< \omega}$ there is some $B \sse X$ such that $\mathbb{G}(B)=\eta$.  Simple rules for computing genus are given in Table \ref{T:key}.  

\begin{table}
\begin{center}
\begin{tabular}{|l|l|} \hline
code & translation\\  \hline
$\langle 1 \ldots \rangle$ & do nothing\\
$\langle 0 \ldots \rangle$ & $(-\infty,\ldots $\\
$\langle \ldots 0,0 \ldots \rangle$ & remove point\\
$\langle \ldots 0,1 \ldots \rangle$ & end interval\\
$\langle \ldots 1,0 \ldots \rangle$ & begin interval\\
$\langle \ldots 1,1 \ldots \rangle$ & add point\\
$\langle \ldots 0 \rangle$ & $\ldots, \infty)$\\
$\langle \ldots 1 \rangle$ & do nothing\\ \hline
\end{tabular}
\end{center}
\caption{A key for assigning forbidden codes to unions of convex sets.}\label{T:key}
\end{table}

To give an example of applying the table, suppose $X=\mathbb{R}$, and $<$ is the usual ordering.  Then the genus of the point $\{0\}$ is $\langle 11 \rangle$, and the genus of the interval $(0,1)$ is $\langle 101\rangle$.  Conversely, to consider all subsets of $\mathbb{R}$ with genus $\langle 11 \rangle$, let $\mathcal{C}$ be all the singletons. Similarly, the collection of all subsets of genus $\langle 101\rangle$ is exactly the set of all infinite convex subsets which are not coinitial or cofinal.

 The assumption that $(X,<)$ is complete was made to give a clear presentation of the genus concept, and is sufficient for this paper.  One can, however, define the genus of $B\sse X$ on other orders by considering the shortest $\eta \in 2^{< \omega}$ which $B$ does not induce, sidestepping the issue of boundary points.  
 
The link between genus and forbidden codes is given in the following theorem.

\begin{theorem}\label{T1}
Suppose that $(X,<)$ is a complete dense linear order without endpoints. If $\eta \in 2^{d+1}$, and $\mathcal{C} = \{C \subseteq X: \mathbb{G}(C) = \eta\}$, then $\mathcal{C}$ is $d$-maximum on $X$ and $\overline{\mathfrak{L}_\mathcal{C}(X)}=\{\eta\}$.

\end{theorem}
\begin{proof}

We provide a sketch of the proof that $\mathcal{C}$ is $d$-maximum, which is very similar to the well-known proof that unions of intervals are maximum.

Let $X_0 \sse X$ be finite, and $\mathcal{C} = \{C \subseteq X: \mathbb{G}(C) = \eta\}$.  Let $a := \text{max }X_0$ and $X_0^a = X_0 \setminus \{a\}$. Define $\mathcal{C}^a = \{C \in \mathcal{C}|^{X_0^a} : C \cup\{a\} \in \mathcal{C}|^{X_0} \And C \in \mathcal{C}|^{X_0}\}$.  By induction on $|X_0|$ and $d$, $\mathcal{C}^a$ is $d-1$ maximum and $\mathcal{C}|^{X_0^a}$ is $d$-maximum.  Then  $|\mathcal{C}|^{X_0}| = |\mathcal{C}^a|+|\mathcal{C}|^{X_0^a}|$, and, by Pascal's identity, $|\mathcal{C}|^{X_0}| = {{|X_0|}\choose{\leq d}}$.

\end{proof}



\section{Results}
\subsection{$\VCm$ and $\VCI$-density}\label{S3}

We now apply the above to achieve our results.  Recall the following definitions. 
\begin{definition}
A formula $\phi({\x},{\y})$ has \emph{VC-density} $\leq$ r for $r \in \mathbb{R}$ if there is $K \in \omega$ such that for every finite $A \sse \M{{\y}}$, $|S_\phi(A)| < K\cdot|A|^r$.  We denote this by $\VCd(\phi) \leq r$. 
\end{definition}
\begin{definition}
A formula $\phi(\x,\y)$ has \emph{$\VCI$-density} $\leq$ r for $r \in \mathbb{R}$ if there is $K \in \omega$ such that for every finite and indiscernible sequence 
$\bar{\bf{b}} =   \langle \bf{b}_i : i < N \rangle  \in \text{M}^{|\y| \cdot N}$, 
 $|S_\phi(range(\bar{\bf{b}}))| < K\cdot N^r$.  We denote this by $\VCId(\phi) \leq r$. 
\end{definition}

The study of $\VC$-density has emerged several times in model theory.  See \cite{AsDoHaMaSt11} for historical remarks.

There has been some study of the fact that frequently $\VCd(\phi)$ is bounded by a simple (and uniform) function of $|\bf{x}|$ \cite{AsDoHaMaSt11,JL10}. When this is true, it justifies the heuristic practice of ``parameter counting" to guess the complexity of set systems. Guingona and Hill showed that in a $dp$-minimal theory $\VCId(\phi) \leq |\x|$.  Thus there is interest in bounding $\VCd(\phi)$ by a function of $\VCId(\phi)$ (obviously $\VCd(\phi) \geq \VCId(\phi)$.)  This may not be possible in general, but we now show a practicable route to achieving it for a given formula.

\begin{definition}
For a set $A \subseteq \M{\y}$ we denote the traces of $\phi(\bf{x},\bf{y})$ on $A$ by
 $$Tr(\phi,A) = \mathcal{P}(S_\phi(A))$$ 
 We refer to the traces of $\phi$ on sets of cardinality $\kappa$ by
$$ Tr_\kappa(\phi) = \bigcup_{A \in [\M{\y}]^\kappa} Tr(\phi,A)$$
 
\end{definition}
The following is easy, but interesting, because the collection of $d$-maximum set systems would seem, \emph{a priori}, to be very diverse.  It also informs Definition \ref{DD}.

\begin{lemma}\label{tts}
Let $\kappa$ be infinite. For each $d\in \omega$, there is some $\mathcal{C} \in Tr_\kappa(\phi)$ such that $\mathcal{C} \text{ is $d$-maximum}$ if and only if for each $n \in \omega$, $n \geq d$, there is $\mathcal{C}_n \in Tr_n(\phi)$ such that $\mathcal{C}_n \text{ is $d$-maximum}$.
\end{lemma}
\begin{proof}
For the right to left direction, it is easily seen that the property of being $d$-maximum is elementary.  Apply the compactness theorem and the saturation of $\monster$.  For left to right, note that if $\mathcal{C}$ is $d$-maximum and $X' \sse X$ with $|X'|=n$, $n\geq d$, then $\mathcal{C}|^{X'}$ is $d$-maximum.
\end{proof}

As a consequence of Lemma \ref{tts}, if $\kappa$ is infinite then there exists $\mathcal{C} \in Tr_\kappa(\phi)$ which is $d$-maximum if and only if there is $\mathcal{C}' \in Tr_{\aleph_0}(\phi)$ such that $\mathcal{C}'$ is $d$-maximum. 

\begin{definition}\label{DD} 
The VC-maximum-dimension of $\phi$ is defined by
$$\VCm(\phi) := \text{max}\{d\in \omega:\exists \mathcal{C} \in Tr_{\aleph_0}(\phi)\, \text{ s.t. } \mathcal{C} \text{ is $d$-maximum}\}$$
\end{definition}

Note that the definition of VC-maximum-dimension does not use model theory.

If $\langle \bf{a}_i \rangle_{i \in I}$ is a sequence of indiscernibles and $(I,<)$ is a complete and dense linear order without endpoints, for $B \sse \langle \bf{a}_i \rangle_{i \in I}$, define $\mathbb{G}(B)$ to be the genus of $\{i \in I : \bf{a}_i \in B \} \sse I$.  With these assumptions, for $m \in \omega$, make the definition 
$$P_m(B) = \{\rho \in 2^m : \exists i_0 < \cdots < i_{m-1} \in I : \bf{a}_{i_j} \in B \iff \rho(j)=1\}$$

If $\mu \in 2^k$, $\eta \in 2^l$ and $l \leq k$, write $\eta \trianglelefteq \mu$ if $\eta$ is a subsequence of $\mu$, meaning that for some order preserving embedding $f:l \rightarrow k$ (where $k$ and $l$ are regarded as ordinals) $\forall i \in l, \mu(f(i)) = \eta(i)$.

We observe that 

\begin{equation} \label{rr}
\forall \rho \in 2^m, \rho \in P_m(B) \iff \mathbb{G}(B) \not \trianglelefteq \rho
\end{equation}
%

Assume in Lemma \ref{LLLLL} that the formula $\phi(\x,\y)$ is NIP.  

\begin{lemma}[Transfer Lemma]\label{LLLLL}
Let $\langle \bf{a}_i \rangle_{i \in I}$ be a sequence of indiscernibles, where $(I,<)$ is a complete and dense linear order without endpoints. Assume that $B \sse \langle \bf{a}_i \rangle_{i \in I}$ is defined by $\phi(\langle \bf{a}_i \rangle_{i \in I},\bf{c})$, $A \sse \langle \bf{a}_i \rangle_{i \in I}$, and $A' \sse A$ can be traced as $A' = A \cap B'$ where $B' \sse \langle \bf{a}_i \rangle_{i \in I}$ is such that $\mathbb{G}(B')=\mathbb{G}(B)$. Then, there exists $\bf{c}' \in \M{\y}$ such that $\phi(A,\bf{c}') = A'$.
\end{lemma}
\begin{proof}
First consider the case in which $A$ is finite.

Let $m = |A|$, and suppose $\bf{a}_{i_0} < \cdots < \bf{a}_{i_{m-1}}$ is an enumeration of $A$.  
Since $\mathbb{G}(B')=\mathbb{G}(B)$, we have $P_m(B)=P_m(B')$ by (\ref{rr}). 
 Let $\rho \in P_m(B')$ be such that for each $j=0,\ldots,m-1$, $\bf{a}_{i_j} \in A' \iff \rho(j)=1$.  Then $\rho \in P_m(B)$, and for some $\bf{a}_{k_0} < \cdots < \bf{a}_{k_{m-1}}$ in $\langle \bf{a}_i \rangle_{i \in I}$, 

$$\monster \models \bigwedge_{j=0}^{m-1} \phi(\bf{a}_{k_j},\bf{c})^{\rho(j)}$$
and thus
$$\monster \models \exists \y \bigwedge_{j=0}^{m-1} \phi(\bf{a}_{k_j},\y)^{\rho(j)}$$
Then, by indiscernibility, 
$$\monster \models \exists \y \bigwedge_{j=0}^{m-1} \phi(\bf{a}_{i_j},\y)^{\rho(j)}$$
and the witnessing $\bf{c}'$ is the desired parameter.

The case in which $A$ is infinite now follows by compactness and the saturation of $\monster$.
\end{proof}

\begin{theorem} \label{bigT}
For any $\phi(\x,\y)$, $\VCm(\phi) = \VCId(\phi)$. 
\end{theorem}
\begin{proof}

First, note that we always have $\VCm(\phi) \geq 0$ and $\VCId(\phi) \geq 0$.

Now suppose $\VCId(\phi) \geq d$, for some positive $d \in \omega$. Let $0 < \epsilon < 1/2$. By compactness, Ramsey's theorem, and saturation of the monster model, there is some indiscernible sequence $\langle \bf{a}_i \rangle_{i \in \mathbb{R}}$ such that $|S_\phi(A)|\geq |A|^{d-\epsilon}$ for arbitrarily large finite $A \sse \langle \bf{a}_i \rangle_{i \in \mathbb{R}}$.  

\textbf{Claim:}  $\exists B \in S_\phi(\langle \bf{a}_i \rangle_{i \in \mathbb{R}})$ with $lg(\mathbb{G}(B)) ={d+1}$.\\
First we argue that there is $B \in S_\phi(\langle \bf{a}_i \rangle_{i \in \mathbb{R}})$ with $lg(\mathbb{G}(B))\geq d+1$. Suppose, to the contrary, that $\forall B \in S_\phi(\langle \bf{a}_i \rangle_{i \in \mathbb{R}})$, we have $lg(\mathbb{G}(B))\leq d$.  
Let $A \sse \langle \bf{a}_i \rangle_{i \in \mathbb{R}}$ be finite, and consider an arbitrary $B \in S_\phi(\langle \bf{a}_i \rangle_{i \in \mathbb{R}})$.  
There are $2^{d+1}-1$ possibilities for $\mathbb{G}(B)$. For any choice of $\mathbb{G}(B)$, by Theorem \ref{T1}, 
$$|\{C \cap A : C \in S_\phi(\langle \bf{a}_i \rangle_{i \in \mathbb{R}}),\, \mathbb{G}(C) = \mathbb{G}(B)\}| \leq {{|A|}\choose{\leq d-1}}$$  
These two facts imply that $|S_\phi(A)| \leq (2^{d+1}-1) \cdot {{|A|}\choose{\leq d-1}} = O(|A|^{d-1})$. Because this holds for any finite $A\sse \langle \bf{a}_i \rangle_{i \in \mathbb{R}}$, the hypothesis on $\langle \bf{a}_i \rangle_{i \in \mathbb{R}}$ is violated.


We now assume $B \in S_\phi(\langle \bf{a}_i \rangle_{i \in \mathbb{R}})$, and $lg(\mathbb{G}(B))\geq d+1$. Without loss of generality, $\phi$ is NIP, because otherwise we have $\VCm(\phi) = \VCId(\phi) = \infty$. Under these assumptions, $\mathbb{G}(B) = n \geq d+1$ for some $n \in \omega$. Inducting on $n$, it follows by compactness and saturation of the monster that there is some $B' \in S_\phi(\langle \bf{a}_i \rangle_{i \in \mathbb{R}})$, with $lg(\mathbb{G}(B'))= d+1$.
 This proves the claim.

Now take $B$ as in the claim.  It follows from the Transfer Lemma that on any finite $A \sse \langle \bf{a}_i \rangle_{i \in \mathbb{R}}$,  $$\{B' \cap A: B'\sse \langle \bf{a}_i \rangle_{i \in \mathbb{R}}, \mathbb{G}(B')=\mathbb{G}(B)\} \in Tr(\phi,A)$$
This implies, by Theorem  \ref{T1}, that $\phi$ traces arbitrarily large $d$-maximum set systems, and, by Lemma \ref{tts}, $Tr_{\aleph_0}(\phi)$ contains a $d$-maximum set system.  Thus $\VCm(\phi) \geq d$.

To show the other direction, suppose $\VCm(\phi) \geq d$.  By compactness, saturation, and Ramsey's theorem (or, alternatively, by Erd\"os-Rado) there is an infinite indiscernible sequence $A = \langle \bf{a}_i \rangle_{i \in \omega}$ on which $\phi$ traces a $d$-maximum set system.  It follows from the definition of $d$-maximum that $S_\phi(A)$ witnesses that $\VCId(\phi) \geq d$.
\end{proof}

It should be noted that Guingona and Hill prove that $\VCId(\phi)$ is equal to several other invariants, among which $\VCm(\phi)$ may obviously be included.

An immediate corollary is the following.

\begin{corollary} \label{C1}
For any formula $\phi(\x,\y)$, $\VCd(\phi) = \VCId(\phi)$ if and only if $\phi$ traces an infinite $d$-maximum set system, where $d = \VCd(\phi)$.
\end{corollary}

This condition is easier to use in practice than a direct appeal to a nonconstructive combinatorial principle such as the Ramsey or Erd\"os-Rado theorem.  We give an algebraic example in the theory of real closed fields (RCF).  Though we make efforts to be precise in the following statements, we are just considering a definable family that results from a polynomial inequality where the coefficients form the parameter set.

Without loss we work in $\mathbb{R}$.  Let $\y$ be a finite sequence of variables. For a given $m \in \omega$, let $Y_m$ denote the set of monomials which occur in the general polynomial of degree $m$ with variables $\y$. Consider a family of polynomials of the form
$$p(c_1,\ldots,c_d,\y) = u_0(\y) + c_1 u_1(\y) + \cdots c_{d} u_d(\y)$$
where for each $i = 0,\ldots,d$, $u_i(\y) \in Y_m$ and for each $i = 1,\ldots,d$, $c_i \in \mathbb{R}$.
Define $\mathcal{C} = \{\text{pos}(p(\bf{c},\y)) : \bf{c} \in \mathbb{R}^d \ \}$, where $\text{pos}(p(\y)) = \{\bf{a} \in \mathbb{R^{|\y|}}: p(\bf{a}) \geq 0\}$. 

Such a $\mathcal{C}$ can clearly be traced by some $\phi(\x,\y) = p(\x,\y) \in L_{ring}$, in $\mathbb{R}\models \text{RCF}$. It is known (see Floyd \cite{Fl89}, section 8.1) that for such a $\phi$, we have $\VCm(\phi) \geq d$.  As it is well-known that $\VCd(\phi) = d$, we have $\VCId(\phi) = \VCd(\phi)$ for polynomial inequalities $\phi$.

Floyd's result is based on an application of Dudley's theorem (\cite{Du99}, Theorem 4.2.1), which can apply to somewhat more general situations (see Ben-David and Litman \cite{BL98}).

Many familiar geometric families, such as circles, ellipses, positive half-spaces, hyperbolas, etc., therefore have $\VCI$-density equal to $\VC$-density.  The above approach notably does not apply to geometric families which are not polynomial definable (in the above sense) such as axis-parallel rectangles, or convex $d$-gons.

\subsection{Stability}
Here we show how to characterize the stability of $\phi$ using the maximum systems in $Tr(\phi,A)$.  For a review of the relevant notions from stability theory see \cite{H97}.

Recall that the \emph{ladder dimension} of a formula $\phi(\x,\y)$ is defined by writing $\LD(\phi) \geq n$ if and only if there are $\bf{a}_0,\ldots,\bf{a}_{n-1}$ in $\M{\x}$ and $\bf{b}_0,\ldots,\bf{b}_{n-1}$ in $\M{\y}$ such that
$\monster \models \phi(\bf{a}_i,\bf{b}_j) \iff i < j$. Finite ladder dimension is equivalent to stability for formulas.  The VC-dimension can be thought of as a generalization of ladder dimension, and in general $\coVC(\phi) \leq \LD(\phi)$, where $\coVC(\phi)$ denotes the $\VC$-dimension of $S_\phi(\M{\y})$ conceived as a set family.

For a set $X$ and $\mathcal{C} \sse \mathcal{P}(X)$, define a graph $\mathcal{G}_\mathcal{C} = (V,E)$ where $V=\mathcal{C}$ and $E(C_1,C_2)$ holds if and only if $|C_1 \Delta C_2| = 1$.  For $C_1, C_2 \in \mathcal{C}$, define $dist_h(C_1,C_2)$ to be the Hamming distance $|C_1 \Delta C_2|$, and $dist_{\mathcal{G}_\mathcal{C}}(C_1,C_2)$ to be the graph distance in $\mathcal{G}_\mathcal{C}$, with $dist_{\mathcal{G}_\mathcal{C}}(C_1,C_2) = \infty$ if $C_1$ and $C_2$ belong to different components.

 The following was proved by Warmuth and Kuzmin \cite{KW}, (Lemma 14).  
\begin{lemma} \label{L3}
Let $X$ be a finite set.  Suppose $\mathcal{C} \sse \mathcal{P}(X)$ is $d$-maximum, and $C_1, C_2 \in \mathcal{C}$. Then
$$dist_h(C_1,C_2) = dist_{\mathcal{G}_\mathcal{C}}(C_1,C_2)$$ 
In particular, $\mathcal{G}_\mathcal{C} $ is connected.
\end{lemma}

The equivalence of $dist_h$ and $dist_{\mathcal{G}_\mathcal{C}}$ is clearly still true when $X$ is infinite, though the graph $\mathcal{G}_\mathcal{C}$ will not be connected in general.  In fact, in many natural maximum set systems (for example, open intervals on a densely ordered set), $\mathcal{G}_\mathcal{C}$ is totally disconnected.  

\begin{theorem}\label{TT}
For $\phi(\x,\y)$ a formula, $\phi$ is stable with $\LD(\phi) \leq n$ if and only if for every $A \sse \M{y}$ and every $\mathcal{C} \in Tr(\phi,A)$ which is $d$-maximum for some $d \in \omega$, for any $C_1,C_2 \in \mathcal{C}$, $|C_1 \setminus C_2| \leq n$. 
\end{theorem}
\begin{proof}
First suppose $\phi(\x,\y)$ is a stable formula with $\LD(\phi) \leq n$, and  $A \sse \M{y}$. 

Let $C_1,C_2 \in \mathcal{C}$, where $\mathcal{C} \in Tr(\phi,A)$ is $d$-maximum for some $d \in \omega$. We will show that $|C_1 \setminus C_2| \leq n$. Note that it suffices to consider the case in which $A$ is finite. Thus we may assume, by Lemma \ref{L3}, that $\mathcal{G}_\mathcal{C}$ is connected.

Let $\{a_1,\ldots,a_k\}  \sse C_1 \setminus C_2$ be a set of distinct elements.  By Lemma \ref{L3}, after possibly reordering, there are $B_k, B_{k-1},\ldots,B_1$ in $\mathcal{C}$, on the path from $C_1$ to $C_2$ in $\mathcal{G}_{\mathcal{C}}$, such that $a_i \in B_j$ iff $i < j$.  Thus $k \leq n$, and consequently $|C_1 \setminus C_2| \leq n$. 

Conversely, suppose $\LD(\phi) > n$.  Let $B = \{\bf{b}_0,\ldots,\bf{b}_{n}\} \sse \M{\x}$ and define $A = \{ \bf{a}_0,\ldots,\bf{a}_{n} \} \sse \M{\y}$ such that
$\monster \models \phi(\bf{b}_i,\bf{a}_j) \iff i < j$.  Put $\mathcal{C} = S_\phi(A)|_B$.  Then by Theorem \ref{T1}, $\mathcal{C}$ is 1-maximum, and $|A \setminus \emptyset| = |A| = n+1$.

\end{proof}

The above theorem shows that much of the nature of Shelah's famous characterization of stable formulas by indiscernibles (see \cite{Sh90}) is already visible at the level of maximum traces. Unspooling the theorem reveals a structural characterization of stable maximum set systems, as we now show.


If $X$ is a set, and $\mathcal{C} \sse \mathcal{P}(X)$, define the \emph{ladder dimension} of $\mathcal{C}$ to be the maximal $n\in \omega$ such that there are $x_1,\ldots,x_n \in X$ and $B_1,\ldots,B_n \in \mathcal{C}$ with $x_i \in B_j$ if and only if $i < j$.  Say that $\mathcal{C}$ is \emph{stable} just in case it has finite ladder dimension.  

If $A \sse X$, define $\mathcal{C}\Delta A = \{C\Delta A : C \in \mathcal{C}\}$.

\begin{lemma} \label{LLL}
If $\mathcal{C} \sse \mathcal{P}(X)$ has $\LD(\mathcal{C})=n$, then for any $A \sse X$, we have $\LD(\mathcal{C}\Delta A)\leq 2n$, and this bound is tight.
\end{lemma}
\begin{proof}
Suppose that $\LD(\mathcal{C}\Delta A) = 2n$ for an integer $n$. Then there exist $x_1,\ldots,x_{2n}$ in $X$ and $C_1,\ldots,C_{2n} \in \mathcal{C}$ such that for all $i,j \leq 2n$, $x_i \in C_j \Delta A$ if and only if $i < j$.  Consider the case in which there are indices ${i_1}<\cdots<{i_n}$ such that for each $k=1,\ldots, n$, $x_{i_k} \not \in A$.  Then these elements, together with an appropriate choice of $C_{j_1},\ldots,C_{j_n}$, witness that $LD(\mathcal{C}) \geq n$.  Now suppose that the opposite holds, namely that there are indices ${i_1}<\cdots<{i_{n+1}}$ such that for each $k=1\ldots n+1$, $x_{i_k} \in A$.  Then $i_k < j$ and $x_{i_k} \not \in C_j$ are equivalent, since both are equivalent to $x_{i_k} \in C_j \Delta A$. Reindexing by $C'_j = C_{2n-j}$ and taking an appropriate $j_1 < \cdots < j_{n}$, we have that $x_{i_k} \in C'_{j_l}$ if and only if $k < l$.  Thus $LD(\mathcal{C}) \geq n$.    A similar argument shows that $LD(\mathcal{C}) \geq n$ in the case in which $LD(\mathcal{C}\Delta A) = 2n+1$ is odd.  This establishes the bound.

To see that the bound is tight, fix $n\in \omega$. Let $X$ be the integers between $-n$ and $n$, inclusive, but not including zero. Let $\mathcal{C} = \{[0,i]\cap X: 0 <i \leq n\}\cup \{[-i,0]\cap X: 0 < i \leq n\}$.  Clearly $\LD(\mathcal{C}) = n$.  But $\LD(\mathcal{C}\Delta [-n,-1]) = 2n$.
\end{proof}

Note that the example in the above proof is $1$-maximum.

\begin{corollary} \label{CC}
Let $\mathcal{C} \sse \mathcal{P}(X)$ be $d$-maximum of ladder dimension $n$.
\begin{enumerate}
\item If $\emptyset \in \mathcal{C}$, then $\mathcal{C} \sse [X]^{\leq n}$.
\item  $\mathcal{C} \Delta B \sse [X]^{\leq 2n}$ for any $B \in \mathcal{C}$. 
\item  $\mathcal{C} \sse [X]^{\leq 2n} \Delta B$ for any $B \in \mathcal{C}$. 
\end{enumerate}
\end{corollary}
\begin{proof}{The claim in (1) is clear from Theorem \ref{TT}.  For the claim in (2) note that $\emptyset \in \mathcal{C} \Delta B$ (because $B \in \mathcal{C}$), and the ladder dimension of $\mathcal{C} \Delta B$ is at most $2n$ by Lemma \ref{LLL}. Therefore $\mathcal{C} \Delta B \sse [X]^{\leq 2n}$ by (1).  Claim (3) follows after applying $\Delta B$ to both sides of the containment in (2).}
\end{proof}


The $2n$ bound in Corollary \ref{CC} part (2) is tight, as the following example shows.  Let $X = \{1,\ldots,2n\}$, and $\mathcal{C} = [X]^{\leq n}$.  Clearly $\LD(\mathcal{C}) = n$.  Now putting $B= \{1,\ldots,n\}$ gives $\mathcal{C}\Delta B$ an element of cardinality $2n$. 

On the other hand, it seems possible that for some $B \sse X$, it holds that $\mathcal{C} \Delta B \sse [X]^{\leq n}$, where the hypotheses are as in Corollary \ref{CC}. However, since the hypotheses admit all finite maximum classes, this conjecture may be too optimistic.  Such a result would clearly be the best possible.

It is also evident from the above that the stable maximum set systems are exactly those maximum set systems $\mathcal{C} \sse \mathcal{P}(X)$ which can be realized as $\mathcal{C} \sse [X]^{m} \Delta B$ for some $m \in \omega$ and $B \sse X$ (because the latter systems are clearly stable).

It would be useful to know whether every $\mathcal{C}$ of ladder dimension $n$ embeds into a $O(n)$-maximum set system $\mathcal{C'}$.  See \cite{BL98} for relevant embedding notions. It is conjectured (see \cite{Fl89}) that the vast majority of set systems are \textit{not} embeddable in maximum systems of the same VC-dimension, prompting the question of whether stable set systems, which are well behaved in so many respects, are also unusual in this way.  Very little is known about model theoretic criteria for when a definable family embeds in a maximum family, other than the easy observation that this is frequently possible in dimension one.

 Many nice properties of maximum set systems, in particular the existence of \textit{compression schemes} (see \cite{KW} for definitions), are inherited under the relation of embedded substructure. Compression schemes emphasize the amount of information needed to represent a $\phi$-type rather than the definability of the representation, and as a consequence they can be used to bound not only the VC-density of a set system, but also the size of the multiplicative constant in the definition of VC-density.  



 

\section{Conclusion}

In model theory, much of the combinatorial content of theories comes from considering formulas restricted to indiscernible sequences. The existence of sequences of indiscernibles is guaranteed by Ramsey's theorem (and compactness), though it is rarely required to exhibit a concrete sequence of indiscernibles.    

When dealing with a certain formula $\phi(\x,\y)$ on a sequence $A = \langle \bf{a}_i \rangle_{i\in I}$, it is a weaker condition to assume that $\phi$ is maximum on $A$ than to assume that $A$ is indiscernible.  However, as we have seen, if $\phi$ is maximum on $A$, that provides ``enough" indiscernibility for some combinatorial notions to manifest.  Namely, $dp$-rank, NIP, and stability can all be understood in terms of the maximum property.  Avoiding the use of indiscernibles has the potential to make these concepts, particularly $dp$-rank, much more accessible to non-model-theorists.

Unlike indiscernible sequences, maximum domains are frequently easily constructible.  In the semilinear case, it follows from the work of Floyd and Dudley that a basic semilinear family will be maximum on a set of points which is in ``general position," for which it is sufficient to take a randomly selected set of points.

Moreover, the similarity of maximum families and formulas on indiscernible sequences provides a point of contact between work done in computational learning theory and model theory, where, especially recently, researchers are pursuing compatible combinatorial goals, but without a common framework.



\bibliographystyle{jflnat}

\section*{Acknowledgements}
Many thanks to the referee and to Vince Guingona for comments on an early version of this paper. This research was partially supported by PSC-CUNY grant \# 64679-00 42.
\\

\end{document}